\documentclass[12pt,a4paper,leqno,verbatim]{amsart}
\newcounter{minutes}
\divide\time by 60
\newcounter{hours}
\multiply\time by 60 \addtocounter{minutes}{-\time}

\usepackage{amssymb}
\usepackage{hyperref}
\usepackage{graphicx}
\usepackage{subfigure}
\date{}
\newfont{\cyrilic}{wncyr10 scaled 1000}

\title[{Generalized complete elliptic integrals}]
{On a function involving generalized complete \\
$(p,q)$- elliptic integrals}

\author[B. A. Bhayo]{Barkat Ali Bhayo}
\address{Department of mathematics, Sukkur IBA University, Sindh, Pakistan}
\address{Faculty of Engineering and natural sciences, Sabanci University, 34956 Tuzla/Istanbul, Turkey}
\email{barkat.bhayo@iba-suk.edu.pk,\quad barkatbhayo@sabanciuniv.edu}

\author[L. Yin]{Li Yin}
\address{Department of Mathematics, Binzhou University, Binzhou City, Shandong Province, 256603, China}
\email{yinli\_79@163.com}

\newcommand{\comment}[1]{}

\swapnumbers
\theoremstyle{plain}

\newtheorem{theorem}[equation]{Theorem}
\newtheorem{lemma}[equation]{Lemma}

\newtheorem{remark}[equation]{Remark}

\newtheorem{proposition}[equation]{Proposition}

\numberwithin{equation}{section}

\begin{document}

\font\fFt=eusm10 
\font\fFa=eusm7  
\font\fFp=eusm5  
\def\K{\mathchoice
{\hbox{\,\fFt K}}
{\hbox{\,\fFt K}}
{\hbox{\,\fFa K}}
{\hbox{\,\fFp K}}}
\def\E{\mathchoice
{\hbox{\,\fFt E}}
{\hbox{\,\fFt E}}
{\hbox{\,\fFa E}}
{\hbox{\,\fFp E}}}

\allowdisplaybreaks

\begin{abstract}
Motivated by the work of Alzer and Richards \cite{ar}, here authors study the monotonicity and convexity properties of the function
$$\Delta _{p,q} (r) = \frac{{E_{p,q}(r)  - \left( {r'} \right)^p K_{p,q}(r) }}{{r^p }} - \frac{{E'_{p,q}(r)  - r^p K'_{p,q}(r) }}{{\left( {r'} \right)^p }},$$
where $K_{p,q}$ and $E_{p,q}$ denote the complete $(p,q)$- elliptic integrals of the first and second kind, respectively.
\end{abstract}

\vspace{.5cm}

\maketitle



\def\thefootnote{}
\footnotetext{ \texttt{\tiny File:~\jobname .tex,
          printed: \number\year-\number\month-\number\day,
          \thehours.\ifnum\theminutes<10{0}\fi\theminutes}
} \makeatletter\def\thefootnote{\@arabic\c@footnote}\makeatother

\section{introduction}

For $0<r<1$ and $r'=\sqrt{1-r^2}$, the Legendre's complete elliptic integrals of the first and the second kind are defined by

$$\left\{\begin{array}{lll} \K=\K(r)=\displaystyle\int_0^{\pi/2}\frac{dt}{\sqrt{1-r^2\sin(t)^2}}=
\displaystyle\int_0^1{\frac{dt}{\sqrt{(1-t^2)(1-r^2t^2)}}},\\

\bigskip

 \E=\E(r)=\displaystyle\int_0^{\pi/2}\sqrt{1-r^2
\sin(t)^2}dt=\displaystyle\int_0^1\sqrt{\frac{1-r^2t^2}{1-t^2}}dt,\\
                      \K(0)=\displaystyle\frac{\pi}{2}=
											\E(0),\, \K(1)=\infty,
											\,\E(1)=0,\\
											\K'=\K'(r)=\K(r')\quad {\rm and }\quad \E'=\E'(r)=\E(r'),
\end{array}\right.$$
respectively. These integrals have played very crucial role in many branches of mathematics, for example, they helps us to find the length of curves and to express the solution of differential equations. The elliptic integrals of the first and the second kind have been extensive interest of the research for several authors, and many results have been established about these integrals in the literature, e.g., 
see \cite{ref2,ref3,ref4,ref5,ref1,ref9,ref11}.
For the monotonicity, convexity properties, asymptotic approximations, functional inequalities of these integrals and their relations with elementary functions, we refer the reader to see e.g., \cite{avvb,ref6,ref7,ref8,ref10,ref11,ref12,ref13,ref14,ref15} and the references therein.

The Gaussian hypergeometric function is defined by
\begin{equation}\label{hypform}
F(a,b;c;z)=~_{2}F_1(a,b;c;z)=\sum_{n\geqslant 0}\frac{(a)_n(b)_n}{(c)_n}\frac{z^n}{n!}, \quad |z|<1.
\end{equation}
Here $(a)_0=1$ for $a\neq 0$, and $(a)_n$ is the Pochhammer symbol $(a)_n=a(a+1)\cdots(a+n-1)$, for $n\in \mathbf{N}$.
The Gaussian hypergeometric function can be represented in the integral form as follows,
\begin{equation}\label{hyp}
F(a,b;c;z)=\frac{\Gamma(c)}{\Gamma(b)(c-b)}\int_0^1{t^{b-1}(1-t)^{c-b-1}(1-zt)^{-a}dt},
\end{equation}
${\rm Re}(c)>{\rm Re}(a)>0, |\arg(1-z)|<\pi$, see \cite{as}.

For $|s|<1/2$ and $0\leq|r|<1$, the complete elliptic integrals of the first and the 
second kind were slightly generalized by Borwein and Borwein \cite{bb} as follows:
\begin{equation}\label{start}{\rm K}_s(r)=F\left(\frac{1}{2}-s,\frac{1}{2}+s;1;r^2\right),\quad
{\rm E}_s(r)=F\left(-\frac{1}{2}-s,\frac{1}{2}+s;1;r^2\right).
\end{equation}
Note that ${\rm K}_0(r)=\K(r)\quad {\rm and}\quad {\rm E}_0(r)=\E(r).$

In order to define the generalized complete $(p,q)$-elliptic integrals of the first and the second kind, we need to define the generalized sine function.

The eigenfunction $\sin_{p,q}$ of the so-called one-dimensional $(p,q)$-Laplacian problem
\cite{dm}
$$-\Delta_p u=-\left(|u'|^{p-2}u'\right)'
=\lambda|u|^{q-2}u,\,u(0)=u(1)=0,\ \ \ p,q>1,$$
is known as the generalized sine function with two parameters $p,q>1$ in the literature 
(see, \cite{bbk,bv,be,egl,t1,t2,t3,t4}), and defined as
the inverse function of
$${\rm arcsin}_{p,q}(x)=\int^x_0(1-t^p)^{-\frac{1}{q}}dt,\quad 0<x<1.$$
Also the generalized $\pi$ is defined as
$$\pi_{p,q}=2{\rm arcsin}_{p,q}(1)=\frac{2}{q}B\left(1-\frac{1}{p},\frac{1}{q}\right),$$
which is the generalized version of the celebrated formula of $\pi$ proved by Salamin \cite{salamin} and Brent \cite{brent} in 1976.
Here $B(.,.)$ denotes the classical beta function.

For all $p,q\in (1,\infty),\,r\in (0,1)$ and $r' = \left( {1 - r^p } \right)^{1/p}$, the generalized complete $(p,q)$-elliptic integrals of the first and the second kind are defined by
\begin{equation}\label{buffer}
\begin{cases}
K_{p,q} (r) = \int_0^{\pi _{p,q} /2} {\left( {1 - r^q \sin_{p,q} ^{q} t} \right)^{1/p-1}} dt,\quad
K'_{p,q}  = K'_{p,q} (r) = K_{p,q} (r'),\\
E_{p,q} (r) = \int_0^{\pi _{p,q} /2} {\left( {1 - r^q \sin_{p,q} ^{q} t} \right)^{1/p} } dt,
\quad E'_{p,q}  = E'_{p,q} (r) = E_{p,q} (r'),
\end{cases}
\end{equation}
respectively.
Applying the integral representation formula \eqref{hyp}, the generalized complete $(p,q)$-elliptic integrals can be expressed in terms of
hypergeometric functions as follows:
\begin{equation}\label{(1.1)}
K_{p,q} (r) = \frac{{\pi _{p,q} }}{2}F\left( {\frac{1}{q},1 - \frac{1}{p};1- \frac{1}{p}+\frac{1}{q};r^p } \right),
\end{equation}
and
\begin{equation}\label{(1.2)}
E_{p,q} (r) = \frac{{\pi _{p,q} }}{2}F\left( {\frac{1}{q}, - \frac{1}{p};1- \frac{1}{p}+\frac{1}{q};r^p } \right),
\end{equation}
see, e.g. \cite{bhayoyin}.
For $p=q$, we write $K_{p,p}=K_p$. Note that $K_2=\K$ and $E_2=\E$. It is worth to mention that
Takeuchi \cite{t2} proved that
$${\rm K}_s(r)=\frac{\pi}{\pi_p}K_p(r^{2/p})\quad {\rm and} \quad
{\rm E}_s(r)=\frac{\pi}{\pi_p}E_p(r^{2/p}),$$
for $|s|<1/2$ and $p=2/(2s+1)$.


In 1998, Anderson, Qiu and Vamanamurthy \cite[Theorem 1.14]{aqv} studied the monotonicity and convexity property of the function
$$f(r)=\frac{\E-r'^2\K}{r^2}\cdot\frac{r'^2}{\E'-r^2\K'}$$
by giving the following theorem.

\begin{theorem}\label{thm1.5} The function $f(r)$ is increasing and convex from $(0,1)$ onto
$(\pi/4,4/\pi)$. In particular,
$$\frac{\pi}{4}<f(r)<\frac{\pi}{4}+\left(\frac{4}{\pi}-\frac{\pi}{4}\right)r$$
for $r\in (0,1)$. These two inequalities are sharp as $r\to 0$, while the second inequality
is also sharp as $r \to 1$.
\end{theorem}

Theorem \ref{thm1.5} has been extended to the case of the functions  \eqref{start} in \cite[Theorem 3.1]{HTZ} and to the case of 
functions $K_{p,p},\, E_{p,p}$ \eqref{buffer} in \cite[Theorem 3.1]{ref1}.

Recently, Alzer and Richards \cite{ar} studied the properties of the additive counterpart
$$\Delta (r) = \frac{{\E - \left( {1 - r^2 } \right)\K}}{{r^2 }} - \frac{{\E' - r^2 \K'}}
{{\left( {1 - r^2 } \right)}}$$
of the above result, and proved the following theorem.

\begin{proposition}\label{alri}
The function $\Delta (r)$ is strictly increasing and strictly convex from $(0,1)$ onto $(\pi/4-1,1-\pi/4)$. Moreover, for all $r\in(0,1)$, one has
\begin{equation}\label{(1.3)}
\frac{\pi }{4} - 1 + \alpha r < \Delta (r) < \frac{\pi }{4} - 1 + \beta r,
\end{equation}
with best possible constants $\alpha=0$ and $\beta=2-\frac{\pi}{2}=0.42920\ldots.$
\end{proposition}

It is natural to extend the result of Alzer and Richards in terms of generalized complete $(p,q)-$ elliptic integrals of the first and second kind. We generalize their function $\Delta$ by
$$\Delta _{p,q} (r) = \frac{{E_{p,q}  - \left( {r'} \right)^p K_{p,q} }}{{r^p }} - \frac{{E'_{p,q}  - r^p K'_{p,q} }}{{\left( {r'} \right)^p }},
$$
and state the following theorem.

\begin{theorem}\label{thm1.1}
For $p,q>1$, the function $
\Delta_{p,q}$ is strictly increasing and strictly convex from $(0,1)$ onto $
\left( {\frac{{\left( {1 - \frac{1}{p}} \right)\pi _{p,q} }}{{2\left( {1 + \frac{1}{q} - \frac{1}{p}} \right)}} - 1,1 - \frac{{\left( {1 - \frac{1}{p}} \right)\pi _{p,q} }}{{2\left( {1 + \frac{1}{q} - \frac{1}{p}} \right)}}} \right)$,
if the following conditions hold:
\begin{enumerate}
\item $2 + \frac{1}{p} + \frac{1}{{p^2 }} \le \frac{5}{p} + \frac{1}{q} < 3 + \frac{1}{{p^2 }},$ \\
\item $\varepsilon(p,q)>0,$
where $$\varepsilon(p,q)=
20 - \frac{{42}}{p} + \frac{6}{q} + \frac{{21}}{{p^2 }} - \frac{2}{{q^2 }} - \frac{{20}}{{pq}} + \frac{9}{{p^2 q}} - \frac{3}{{p^3 }} - \frac{1}{{p^3 q}}.$$
\end{enumerate}

Moreover, for all $r\in(0,1)$, we have
\begin{equation}\label{(1.4)}
{\frac{{\left( {1 - \frac{1}{p}} \right)\pi _{p,q} }}{{2\left( {1 + \frac{1}{q} - \frac{1}{p}} \right)}} - 1}
+ \alpha_1r < \Delta_{p,q} (r) <
{\frac{{\left( {1 - \frac{1}{p}} \right)\pi _{p,q} }}{{2\left( {1 + \frac{1}{q} - \frac{1}{p}} \right)}} - 1}
+ \beta_1 r,
\end{equation}
with best possible constants $\alpha_1=0$ and $\beta_1=
{2 - \frac{{\left( {1 - \frac{1}{p}} \right)\pi _{p,q} }}{{\left( {1 + \frac{1}{q} - \frac{1}{p}} \right)}}}.$
\end{theorem}

\begin{theorem}\label{thm1.2}
For all $r,s\in(0,1)$ and $p,q>1$ satisfying the conditions (1) and (2) given in the above theorem, then we have 
\begin{equation}\label{(1.5)}
{\frac{{\left( {1 - \frac{1}{p}} \right)\pi _{p,q} }}{{2\left( {1 + \frac{1}{q} - \frac{1}{p}} \right)}} - 1}
 < \Delta _{p,q} (rs) - \Delta _{p,q} (r) - \Delta _{p,q} (s) <
{1 - \frac{{\left( {1 - \frac{1}{p}} \right)\pi _{p,q} }}{{2\left( {1 + \frac{1}{q} - \frac{1}{p}} \right)}}}.
\end{equation}
\end{theorem}

\section{Lemmas}
In this section we give few lemmas which will be used in the proof of the theorems. Moreover, we will use same method for proving our theorems as it is applied in \cite{ar}.

\begin{lemma}\label{lem2.1} Write
$$
H_{a,b} (r) = \frac{\pi _{1/b,1/a} }{2r^{1/b} }\left[ F\left( {a, - b;1+a-b;r^{1/b} } \right) -
\left( 1-r^{1/b} \right) F\left( {a,1 - b;1+a-b;r^{1/b} } \right) \right].$$

For $a,b,r\in(0,1)$, we have
\begin{equation}\label{(2.2)}
H_{a,b} (r) = \frac{(1 - b)\pi _{1/b,1/a} }{2(1+a-b)}F\left( a,1 - b;2+a-b;r^{1/b}  \right).
\end{equation}
\end{lemma}

\begin{proof}
By using the formula \eqref{hypform}, we obtain
$$
\begin{array}{l}
 H_{a,b} (r) =\displaystyle\frac{\pi _{1/b,1/a} }{2r^{1/b} }\left[ {\sum\limits_{n = 0}^\infty  {\frac{{(a)_n ( - b)_n }}{{(1+a-b)_n{n!}  }}r^{n/b} }  - \left( {1 - r^{1/b} } \right)\sum\limits_{n = 0}^\infty  {\frac{{(a)_n (1 - b)_n }}{{(1+a-b)_n{n!} }}r^{n/b} } } \right] \\
  =\displaystyle\frac{{\pi _{1/b,1/a} }}{{2r^{1/b} }}\left[ {\sum\limits_{n = 0}^\infty  {\left( {\frac{{(a)_n ( - b)_n }}
	{{(1+a-b)_n{n!} }} - \frac{{(a)_n (1 - b)_n }}{{(1+a-b)_n{n!} }}} \right)r^{n/b} }  + \sum\limits_{n = 0}^\infty  {\frac{{(a)_n (1 - b)_n }}{{(1+a-b)_n{n!} }}r^{(n + 1)/b} } } \right] \\
  =\displaystyle\frac{{\pi _{1/b,1/a} }}{2} \sum\limits_{n = 0}^\infty  {\left( {\frac{{(a)_{n + 1} ( - b)_{n + 1} }}{{(1+a-b)_{n+1}{(n+1)!} }} - \frac{{(a)_{n + 1} (1 - b)_{n + 1} }}{{(1+a-b)_{n+1}{(n+1)!} }}} \right)r^{n/b} }     \\
	\qquad +\displaystyle\frac{{\pi _{1/b,1/a} }}{2}\sum\limits_{n = 0}^\infty  {\frac{{(a)_n (1 - b)_n }}{{(1+a-b)_{n}{n!} }}r^{n/b} }\\
  =\displaystyle\frac{{\pi _{1/b,1/a}} }{2}\left[ {\sum\limits_{n = 0}^\infty  {\frac{{(a)_n (1 - b)_n }}{{(1+a-b)_{n+1}{(n+1)!} }}\cdot \xi_{a,n} \cdot r^{n/b} } } \right] \\
  =\displaystyle\frac{{(1 - b)\pi _{1/b,1/a} }}{2}\sum\limits_{n = 0}^\infty  {\frac{{(a)_n (1 - b)_n }}{{n!(1+a-b)_{n+1} }}r^{n/b} }  \\
  =\displaystyle\frac{{(1 - b)\pi _{1/b,1/a} }}{2(1+a-b)}F\left( {a,1 - b;2+a-b;r^{1/b} } \right), \\
 \end{array}$$
where
$$\xi_{a,n} = (n + a)( -b) - (a + n)(1 - b + n)
	+ (n + 1)(n+1+a-b)=(1-b)(n+1).$$
\end{proof}


\begin{lemma}\label{lem2.4}
For $p,q>1$ and $r\in(0,1)$, we have
\begin{equation}\label{(2.7)}
H_{1/q,1/p} (0) = \frac{{\left( {1 - \frac{1}{p}} \right)\pi _{p,q} }}{2\left(1 + \frac{1}{q} - \frac{1}{p}\right)},\qquad H_{1/q,1/p} (1) = 1.
\end{equation}
\end{lemma}
\begin{proof}
By definition, it is easy to see that $$
H_{1/q,1/p} (0) = \frac{{\left( {1 - \frac{1}{p}} \right)\pi _{p,q} }}{2\left(1 + \frac{1}{q} - \frac{1}{p}\right)}F\left( {\frac{1}{q},1 - \frac{1}{p};2+\frac{1}{q} - \frac{1}{p};0} \right) =  \frac{{\left( {1 - \frac{1}{p}} \right)\pi _{p,q} }}{2\left(1 + \frac{1}{q} - \frac{1}{p}\right)}.
$$
Again, by using the following identity
\begin{equation}\label{(2.8)}
F\left( {\alpha ,\beta ;\gamma ;1} \right) = \frac{{\Gamma (\gamma )\Gamma (\gamma  - \alpha  - \beta )}}{{\Gamma \left( {\gamma  - \alpha } \right)\Gamma \left( {\gamma  - \beta } \right)}},
\end{equation}
(see, \cite[p. 153]{wg}), we get
\begin{align*}
 H_{1/q,1/p} (1) = &\frac{{\left( {1 - \frac{1}{p}} \right)\pi _{p,q} }}{2\left(1 + \frac{1}{q} - \frac{1}{p}\right)}F\left( {\frac{1}{q},1 - \frac{1}{p};2+\frac{1}{q} - \frac{1}{p};1} \right) \\
  = &\frac{{\left( {1 - \frac{1}{p}} \right)\pi _{p,q} }}{2\left(1 + \frac{1}{q} - \frac{1}{p}\right)}
  \frac{{\Gamma \left(2+ \frac{1}{q} - \frac{1}{p}\right)\Gamma (1)}}{{\Gamma \left( {2 - \frac{1}{p}} \right)\Gamma \left( {1 + \frac{1}{q}} \right)}} \\
  =&\frac{{\pi _{p,q} }}{2}\frac{{\Gamma \left( {1 + \frac{1}{q} - \frac{1}{p}} \right)}}{{\Gamma \left( {1 - \frac{1}{p}} \right)\Gamma \left( {1 + \frac{1}{q}} \right)}} \\
  =& 1.\\
 \end{align*}
\end{proof}

\begin{lemma}\label{lem2.3}
For $p,q>1$ and $r\in(0,1)$, we have
\begin{equation}\label{(2.5)}
\begin{array}{l}
 \left( {2 - \frac{1}{p}} \right)F\left( {1 + \frac{1}{q},3 - \frac{1}{p};4 + \frac{1}{q} - \frac{1}{p};1 - r^p } \right)\\
  =\left(3 + \frac{1}{q} - \frac{1}{p}\right)F\left( {1 + \frac{1}{q},2 - \frac{1}{p};3 + \frac{1}{q} - \frac{1}{p};1 - r^p } \right) \\
  - \left( {1 + \frac{1}{q}} \right)F\left( {1 + \frac{1}{q},2 - \frac{1}{p};4 + \frac{1}{q} - \frac{1}{p};1 - r^p } \right). \\
 \end{array}
\end{equation}
\end{lemma}
\begin{proof}
Utilizing the contiguous relation (See \cite[Equation 26]{pbm})
$$
(\sigma  - \rho )F\left( {\alpha ,\rho ;\sigma  + 1;z} \right) = \sigma F\left( {\alpha ,\rho ;\sigma ;z} \right) - \rho F\left( {\alpha ,\rho  + 1;\sigma  + 1;z} \right),
$$
and letting $
\sigma  = 3 + \frac{1}{q} - \frac{1}{p},\alpha  = 1 + \frac{1}{q},\rho  = 2 - \frac{1}{p},z = 1 - r^p
$, we get the identity  ~\eqref{(2.5)}.
\end{proof}

\section{Proof of the main result}
\vspace{.5cm}
\noindent{\bf Proof of Theorem \ref{thm1.1}.}
Using formulas ~\eqref{(1.1)},  ~\eqref{(1.2)} and letting $a=1/q,\,b=1/p$ in \eqref{(2.2)}, we have
\begin{align*}
 \Delta _{p,q} (r) =& H_{1/q,1/p} (r) - H_{1/q,1/p} (r') \\
  = &\frac{{\left( {1 - \frac{1}{p}} \right)\pi _{p,q} }}{2\left(1 + \frac{1}{q} - \frac{1}{p}\right)}\left[ {F\left( {\frac{1}{q},1 - \frac{1}{p};2 + \frac{1}{q} - \frac{1}{p};r^p } \right) - F\left( {\frac{1}{q},1 - \frac{1}{p};2 + \frac{1}{q} - \frac{1}{p};1 - r^p } \right)} \right]. \\
\end{align*}
Applying the following derivative formula$$
\frac{d}{{dr}}F\left( {a,b;c;r} \right) = \frac{{ab}}{c}F\left( {a + 1,b + 1;c + 1;r} \right),
$$
we obtain
$$\Delta '_{p,q} (r) =
 \eta _{p,q} r^{p - 1} \left[ {F\left( {1 + \frac{1}{q},2 - \frac{1}{p};3 + \frac{1}{q} - \frac{1}{p};r^p } \right) + F\left( {1 + \frac{1}{q},2 - \frac{1}{p};3 + \frac{1}{q} - \frac{1}{p};\left( {r'} \right)^p } \right)} \right],$$
and
$$
\begin{array}{l}
 \frac{1}{{\eta _{p,q} r^{p - 2} }}\Delta ''_{p,q} (r) =  \\
 (p - 1)\left[ {F\left( {1 + \frac{1}{q},2 - \frac{1}{p};3 + \frac{1}{q} - \frac{1}{p};r^p } \right) + F\left( {1 + \frac{1}{q},2 - \frac{1}{p};3 + \frac{1}{q} - \frac{1}{p};\left( {r'} \right)^p } \right)} \right] \\
  + \frac{{p\left( {1 + \frac{1}{q}} \right)\left( {2 - \frac{1}{p}} \right)}}{{\left( {3 + \frac{1}{q} - \frac{1}{p}} \right)}}r^{p}\left[ {F\left( {2 + \frac{1}{q},3 - \frac{1}{p};4 + \frac{1}{q} - \frac{1}{p};r^p } \right) - F\left( {2 + \frac{1}{q},3 - \frac{1}{p};4 + \frac{1}{q} - \frac{1}{p};\left( {r'} \right)^p } \right)} \right], \\
 \end{array}
$$ where
$$\eta _{p,q}  = \frac{{\frac{p}{q}\left( {1 - \frac{1}{p}} \right)^2 \pi _{p,q} }}{{2\left( {1 + \frac{1}{q} - \frac{1}{p}} \right)\left( {2 + \frac{1}{q} - \frac{1}{p}} \right)}}.$$
By utilizing the following identity
$$(1 - z)F\left( {a + 1,b + 1;a + b + 1;z} \right) = F\left( {a,b;a + b + 1;z} \right)$$
(see \cite{ar}) and letting $a=1+\frac{1}{q},b=2-\frac{1}{p},z=1-r^p$, we get
$$
\begin{array}{l}
 \frac{1}{{\eta _{p,q} r^{p - 2} }}\Delta ''_{p,q} (r) =  \\
 (p - 1)\left[ {F\left( {1 + \frac{1}{q},2 - \frac{1}{p};3 + \frac{1}{q} - \frac{1}{p};r^p } \right) - F\left( {1 + \frac{1}{q},2 - \frac{1}{p};3 + \frac{1}{q} - \frac{1}{p};\left( {r'} \right)^p } \right)} \right] \\
  + \frac{{p\left( {1 + \frac{1}{q}} \right)\left( {2 - \frac{1}{p}} \right)}}{{\left( {3 + \frac{1}{q} - \frac{1}{p}} \right)}}r^p F\left( {2 + \frac{1}{q},3 - \frac{1}{p};4 + \frac{1}{q} - \frac{1}{p};r^p } \right) \\
  - \frac{{p\left( {1 + \frac{1}{q}} \right)\left( {2 - \frac{1}{p}} \right)}}{{\left( {3 + \frac{1}{q} - \frac{1}{p}} \right)}}F\left( {1 + \frac{1}{q},2 - \frac{1}{p};4 + \frac{1}{q} - \frac{1}{p};\left( {r'} \right)^p } \right). \\
 \end{array}
$$
Now by Lemma~\ref{lem2.3} we have
$$
\begin{array}{l}
 \frac{1}{{\eta _{p,q} r^{p - 2} }}\Delta ''_{p,q} (r) =  \\
 (p - 1)F\left( {1 + \frac{1}{q},2 - \frac{1}{p};3 + \frac{1}{q} - \frac{1}{p};r^p } \right) +(p-1)F\left( {1 + \frac{1}{q},2 - \frac{1}{p};3 + \frac{1}{q} - \frac{1}{p};\left( {r'} \right)^p } \right) \\
  + \frac{{p\left( {1 + \frac{1}{q}} \right)\left( {2 - \frac{1}{p}} \right)}}{{\left( {3 + \frac{1}{q} - \frac{1}{p}} \right)}}r^p F\left( {2 + \frac{1}{q},3 - \frac{1}{p};4 + \frac{1}{q} - \frac{1}{p};r^p } \right) \\
  - p\left( {2 - \frac{1}{p}} \right)F\left( {1 + \frac{1}{q},2 - \frac{1}{p};3 + \frac{1}{q} - \frac{1}{p};\left( {r'} \right)^p } \right) \\
  + \frac{{p\left( {2 - \frac{1}{p}} \right)^2 }}{{\left( {3 + \frac{1}{q} - \frac{1}{p}} \right)}}F\left( {1 + \frac{1}{q},3 - \frac{1}{p};4 + \frac{1}{q} - \frac{1}{p};\left( {r'} \right)^p } \right) \\
  = (p - 1)F\left( {1 + \frac{1}{q},2 - \frac{1}{p};3 + \frac{1}{q} - \frac{1}{p};r^p } \right) \\
  + \frac{{p\left( {1 + \frac{1}{q}} \right)\left( {2 - \frac{1}{p}} \right)}}{{\left( {3 + \frac{1}{q} - \frac{1}{p}} \right)}}r^p F\left( {2 + \frac{1}{q},3 - \frac{1}{p};4 + \frac{1}{q} - \frac{1}{p};r^p } \right) \\
  - \left[ {p\left( {2 - \frac{1}{p}} \right) - (p - 1)} \right]F\left( {1 + \frac{1}{q},2 - \frac{1}{p};3 + \frac{1}{q} - \frac{1}{p};\left( {r'} \right)^p } \right) \\
  + \frac{{p\left( {2 - \frac{1}{p}} \right)^2 }}{{\left( {3 + \frac{1}{q} - \frac{1}{p}} \right)}}F\left( {1 + \frac{1}{q},3 - \frac{1}{p};4 + \frac{1}{q} - \frac{1}{p};\left( {r'} \right)^p } \right). \\
 \end{array}
$$
Considering the condition (1) and (2) we conclude that
$$
\begin{array}{l}
 \frac{1}{{\eta _{p,q} r^{p - 2} }}\Delta ''_{p,q} (r) \\
  \ge (p - 1) - pF\left( {1 + \frac{1}{q},2 - \frac{1}{p};3 + \frac{1}{q} - \frac{1}{p};\left( {r'} \right)^p } \right) \\
  + \frac{{p\left( {2 - \frac{1}{p}} \right)^2 }}{{\left( {3 + \frac{1}{q} - \frac{1}{p}} \right)}}F\left( {1 + \frac{1}{q},3 - \frac{1}{p};4 + \frac{1}{q} - \frac{1}{p};\left( {r'} \right)^p } \right) \\
	
  = (p - 1)\left[ {1 + \sum\limits_{n = 0}^\infty  {\frac{{\left( {1 + \frac{1}{q}} \right)_n \left( {2 - \frac{1}{p}} \right)_n}}
	{{\left( {3 + \frac{1}{q} - \frac{1}{p}} \right)_{n+1}}}\left( {n - \frac{{\frac{1}{q} + \frac{3}{p} - 1 - \frac{1}{{p^2 }}}}{{1 - \frac{1}{p}}}} \right)\frac{{\left( {r'} \right)^{pn} }}{{n!}}} } \right] \\

  > (p - 1)\left[ {1 + \sum\limits_{n = 0}^1 {\frac{{\left( {1 + \frac{1}{q}} \right)_n \left( {2 - \frac{1}{p}} \right)_n}}{{\left( {3 + \frac{1}{q} - \frac{1}{p}} \right)_{n+1}}}\left( {n - \frac{{\frac{1}{q} + \frac{3}{p} - 1 - \frac{1}{{p^2 }}}}{{1 - \frac{1}{p}}}} \right)\frac{{\left( {r'} \right)^{pn} }}{{n!}}} } \right] \\
	
  = (p - 1)\left[ {\frac{{\varepsilon (p,q)}}{{\left( {1 - \frac{1}{p}} \right)\left( {3 + \frac{1}{q} - \frac{1}{p}} \right)\left( {4 + \frac{1}{q} - \frac{1}{p}} \right)}} + \frac{{\left( {1 + \frac{1}{q}} \right)\left( {2 - \frac{1}{p}} \right)\left( {\frac{4}{p} + \frac{1}{q} - 2 - \frac{1}{{p^2 }}} \right)r^p }}{{\left( {1 - \frac{1}{p}} \right)\left( {3 + \frac{1}{q} - \frac{1}{p}} \right)\left( {4 + \frac{1}{q} - \frac{1}{p}} \right)}}} \right] \\
  > 0.
 \end{array}
$$
Now we conclude that $\Delta '_{p,q} (r) $ is strictly increasing on $(0,1)$, because $\Delta ''_{p,q} (r)>0 $. Hence, $\Delta '_{p,q} (r)> \Delta '_{p,q}(0).$
By L'H\^{o}spital rule, we get
\begin{align*}
 \Delta _{p,q} ^\prime  (0) =& \mathop {\lim }\limits_{r \to 0^ +  } \frac{{\Delta _{p,q} (r) - \Delta _{p,q} (0)}}{{r - 0}} \\
  =& \mathop {\lim }\limits_{r \to 0^ +  }\left[ \frac{{H_{1/q,1/p} (r) - H_{1/q,1/p} (0)}}{{r }} - \frac{{H_{1/q,1/p} (r') - H_{1/q,1/p} (1)}}{{r }}\right] \\
  =& H'_{1/q,1/p} (0) - \mathop {\lim }\limits_{x \to 1^ +  } \frac{{H_{1/q,1/p} (x) - 1}}{{\left( {1 - x^p } \right)^{1/p} }} \\
  = &\mathop {\lim }\limits_{x \to 1^ +  } H'_{1/q,1/p} (x)\frac{{\left( {1 - x^p } \right)^{1 - 1/p} }}{{x^{p - 1} }} \\
  = &0,
\end{align*}
where we apply $$
\begin{array}{l}
 H'_{1/q,1/p} (0)  \\
 =\displaystyle\frac{{\frac{p}{q}\left( {1 - \frac{1}{p}} \right)^2 r^{p - 1}\pi_{p,q} }}{{2\left( {1 + \frac{1}{q} - \frac{1}{p}} \right)\left( {2 + \frac{1}{q} - \frac{1}{p}} \right)}}F\left( {1 + \frac{1}{q},2 - \frac{1}{p};3 + \frac{1}{q} - \frac{1}{p};r^p } \right)\left| {r = 0}. \right. \\
 \end{array}$$
This implies that $\Delta_{p,q} (r) $ is strictly increasing on $(0,1)$.
Define $$
M_{p,q} (r) = \frac{{\Delta _{p,q} (r) - \Delta _{p,q} (0)}}{{r - 0}}.
$$
Since $\Delta_{p,q} (r) $ is strictly convex on $(0,1)$, it follows that $M_{p,q} (r)$ is strictly increasing on $(0,1)$. This leads to
\begin{equation}\label{(2.9)}
M_{p,q} (0)<M_{p,q} (r)<M_{p,q} (1).
\end{equation}
Using Lemma~\ref{lem2.4}, we have $M_{p,q} (0)=0$ and
$$M_{p,q} (1) = \Delta _{p,q} (1) - \Delta _{p,q} (0) = 2H_{1/q,1/p} (1) - 2H_{1/q,1/p} (0) =
{2 - \frac{{\left( {1 - \frac{1}{p}} \right)\pi _{p,q} }}{{\left( {1 + \frac{1}{q} - \frac{1}{p}} \right)}}}.
$$
So, the formula ~\eqref{(2.9)} implies the double inequalities in ~\eqref{(1.4)}.
This completes the proof.$\hfill\square$

\vspace{.4cm}
\noindent{\bf Proof of Theorem \ref{thm1.2}.}
Let $$
\lambda _{p,q} (r,s) = \Delta _{p,q} (rs) - \Delta _{p,q} (r) - \Delta _{p,q} (s).
$$
Simple computation yields $$
\frac{\partial }{{\partial r}}\lambda _{p,q} (r,s) = s\Delta '_{p,q} (rs) - \Delta _{p,q} ^\prime  (r)
$$ and $$
\frac{{\partial ^2 }}{{\partial r\partial s}}\lambda _{p,q} (r,s) = \Delta '_{p,q} (rs) + rs\Delta _{p,q} ^{\prime \prime } (rs).
$$
Considering the results of Theorem \ref{thm1.1}, we have $\frac{{\partial ^2 }}{{\partial r\partial s}}\lambda _{p,q} (r,s)>0$. So, the function
$\frac{\partial }{{\partial r}}\lambda _{p,q} (r,s)$ is strictly increasing with respect to $s$. Thus, $$
\frac{\partial }{{\partial r}}\lambda _{p,q} (r,s) < \frac{\partial }{{\partial r}}\lambda _{p,q} (r,s)\left| {_{s = 1} } \right. = 0.
$$
It follows that $
r \mapsto \lambda _{p,q} (r,s)
$ is strictly decreasing which lead to $$
 - \Delta _{p,q} (1) = \lambda _{p,q} (1,s) < \lambda _{p,q} (r,s) < \lambda _{p,q} (0,s) =-\Delta _{p,q} (s)<  - \Delta _{p,q} (0).
$$ where we apply
$$
\Delta _{p,q} (0) = H_{1/q,1/p} (0) - H_{1/q,1/p} (1) = \frac{{\left( {1 - \frac{1}{p}} \right)\pi _{p,q} }}{{2\left( {1 + \frac{1}{q} - \frac{1}{p}} \right)}} - 1$$
and
$$
\Delta _{p,q} (1) = H_{1/q,1/p} (1) - H_{1/q,1/p} (0) = 1 - \frac{{\left( {1 - \frac{1}{p}} \right)\pi _{p,q} }}{{2\left( {1 + \frac{1}{q} - \frac{1}{p}} \right)}}.$$
This completes the proof.$\hfill\square$

\begin{remark}\rm
When $p=q=2$, then it is easy to observe that Theorem \ref{thm1.2} coincides with Proposition \ref{alri}.
\end{remark}

\end{document}